\documentclass[12pt,oneside]{amsart}
\usepackage[margin=1in]{geometry}
\usepackage{amsmath}
\usepackage{amsthm}
\usepackage{amsfonts}
\usepackage{amssymb}
\usepackage[hidelinks]{hyperref}
\usepackage{bm}
\usepackage{tikz}
\usetikzlibrary{cd}
\usetikzlibrary{decorations.markings}
\tikzset{degil/.style={
            decoration={markings,
            mark= at position 0.5 with {
                  \node[transform shape] (tempnode) {$\backslash$};
                  }
              },
              postaction={decorate}
            }
}

\numberwithin{equation}{section}
\numberwithin{figure}{section}

\theoremstyle{plain}
\newtheorem{thm}[equation]{Theorem}
\newtheorem*{FundClaim*}{Fundamental Claim}
\newtheorem{lemma}[equation]{Lemma}

\theoremstyle{definition}

\newtheorem{remark}[equation]{Remark}

\newcommand{\Q}{\ensuremath \mathbb{Q}}
\newcommand{\R}{\ensuremath \mathbb{R}}

\newcommand{\Z}{\ensuremath \mathbb{Z}}
\newcommand{\N}{\ensuremath \mathbb{N}}

\begin{document}

\title[The $\Z$-Divisor Graph on ${\Q\cap[1,7]}$ is Planar and 3-Colorable]{The $\Z$-Divisor Graph on $\Q\cap[1,7]$ is Planar and 3-Colorable}

\author{Michael D.\ Walton, Zachary D.\ Walton}
\address{Department of Mathematics, Baylor University, Waco, TX 76706, USA}
\email{michael\_walton2@baylor.edu}

\keywords{Divisor Graph, Planar Graph}

\begin{abstract}
    We prove that a divisor graph on an interval of integers is a planar graph whenever the interval $[n,m]$ satisfies $m\leq 7n$, and that $7$ is the greatest such real number for which we can guarantee that the graph is planar. We prove these graphs are 3-colorable.
\end{abstract}

\maketitle

\section{Introduction}\label{Section:Introduction}

A \emph{divisor graph} $G(S)$ of a subset $S\subseteq \Z$ is the simple graph $(V,E)$ where $V=S$ and $uv\in E$ if and only if $u\neq v$ and either $u|v$ or $v|u$. Divisor graphs were formally introduced by Santhosh and Singh in 2000 \cite{Singh_Santhosh_preprint}, and have been discussed further in several other papers, such as \cite{Singh_Santhosh} and \cite{Chartrand_etal}. Other papers, such as \cite{Maan_Sehgal_Malik1} and \cite{Maan_Sehgal_Malik2}, have looked at divisor graphs over polynomial rings.

In this paper, we consider divisor graphs of the form $G(\{n, n+1, n+2, \ldots, m-1,m\})$ for $n,m\in \Z_{>0}$, denoted hereafter as $G[n,m]$. We consider when $G[n,m]$ is a planar graph. For example, $G[1,14]$ is a planar graph, but $G[1,15]$ is non-planar as it contains $K_{3,3}$ as a minor, as shown below.

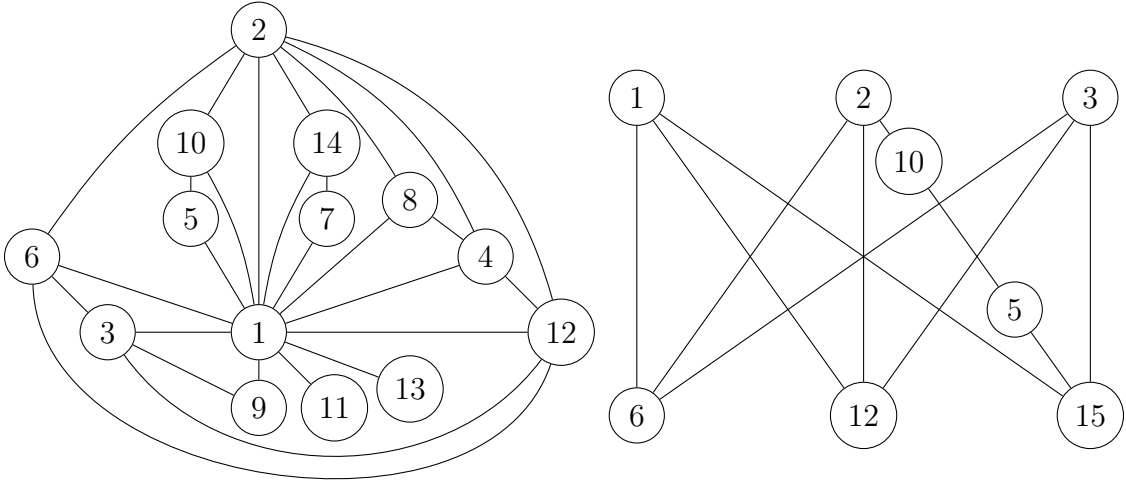
\begin{center}
\begin{tikzpicture}[scale=1, every node/.style={circle, draw, minimum size=5mm}]
\node (v1) at (0,-1) {1};
\node (v2) at (0,3) {2};
\node (v3) at (-2,-1) {3};
\node (v4) at (3,0) {4};
\node (v5) at (-0.9,0.5) {5};
\node (v6) at (-3,0) {6};
\node (v7) at (0.9,0.5) {7};
\node (v8) at (2,0.75) {8};
\node (v9) at (0,-2) {9};
\node (v10) at (-0.9,1.5) {10};
\node (v11) at (1,-2) {11};
\node (v12) at (4,-1) {12};
\node (v13) at (2,-1.75) {13};
\node (v14) at (0.9,1.5) {14};

\draw (v1) -- (v2);
\draw (v1) -- (v3);
\draw (v1) -- (v4);
\draw (v1) -- (v5);
\draw (v1) -- (v6);
\draw (v1) -- (v7);
\draw (v1) -- (v8);
\draw (v1) -- (v9);
\draw[bend right=10] (v1) to (v10);
\draw (v1) -- (v11);
\draw (v1) -- (v12);
\draw (v1) -- (v13);
\draw[bend left=10] (v1) to (v14);

\draw[bend left=20] (v2) to (v4);
\draw[bend right=10] (v2) to (v6);
\draw[bend left=10] (v2) to (v8);
\draw (v2) -- (v10);
\draw[bend left=30] (v2) to (v12);
\draw (v2) -- (v14);

\draw (v3) -- (v6);
\draw (v3) -- (v9);
\draw[bend right=54] (v3) to (v12);

\draw (v4) -- (v8);
\draw (v4) -- (v12);

\draw (v5) -- (v10);

\draw[bend right=80] (v6) to (v12);

\draw (v7) -- (v14);

\node (b1) at (5,2.1) {1};
\node (b2) at (8,2.1) {2};
\node (b3) at (11,2.1) {3};
\node (b6) at (5,-2.1) {6};
\node (b12) at (8,-2.1) {12};
\node (b15) at (11,-2.1) {15};
\node (b10) at (8.6,1.26) {10};
\node (b5) at (10,-0.7) {5};

\draw (b1) -- (b6);
\draw (b1) -- (b12);
\draw (b1) -- (b15);
\draw (b2) -- (b6);
\draw (b2) -- (b12);
\draw (b2) -- (b10);
\draw (b10) -- (b5);
\draw (b5) -- (b15);
\draw (b3) -- (b6);
\draw (b3) -- (b12);
\draw (b3) -- (b15);
\end{tikzpicture}
\end{center}

In this paper, we consider the ratio $m/n$ where $m$ is the greatest integer for which $G[n,m]$ is planar, and we prove that the infimum of all such rational numbers $m/n$ is 7. To that end, we define the \emph{rational divisor graph} $Q(S)$ of a subset $S\subseteq \Q$ to be the simple graph $(V,E)$ where $V=S$ and $uv\in E$ if and only if $u\neq v$ and either $u=nv$ or $v=nu$ for some $n\in \N=\{1, 2, 3, \ldots\}$. If $S = [a,b]\cap\Q$ or $S=(a,b)\cap \Q$, then we denote $Q(S)$ as $Q[a,b]$ and $Q(a,b)$, respectively. Note that if we rename every vertex $k$ of a graph $G[n,m]$ to be $k/n$, then we can naturally associate the graph $G[n,m]$ to a subgraph of the rational divisor graph $Q[1,m/n]$. We will prove that $Q[1,7]$ is planar, but $Q[1,7+\epsilon)$ is non-planar for every $\epsilon>0$. We will further prove that $Q[1,7]$ is 3-colorable.

\section{Results}

We first prove a brief lemma.

\begin{lemma}\label{Lem:Max_Divisor}
    If $q_1q_2$ is an edge in $Q[1,r]$ with $q_1<q_2$, then $q_2 = kq_1$ for some $k\in \Z\cap[1,r]$.
\end{lemma}
\begin{proof}
    We prove the contrapositive. If $q_2=kq_1$ with $k>r$, then as $1\leq q_1$, $r<k\leq kq_1=q_2$, so $q_2$ is not a vertex in $Q[1,r]$.
\end{proof}

As a result of Lemma \ref{Lem:Max_Divisor}, two vertices $q_1<q_2$ in $Q[1,7]$ are connected if and only if $q_2=kq_1$ for $k\in\{2,3,4,5,6,7\}$. We will first prove that $Q[1,7)$ is planar and 3-colorable.

\begin{thm}
    $Q[1,7)$ is planar and 3-colorable.
\end{thm}
\begin{proof}
    Note that if every connected component of a countable graph $G$ is planar (or 3-colorable), then $G$ is planar (or 3-colorable) as we can embed each component into disjoint open discs in the plane. Thus it suffices to prove the component containing $1$ is planar (or 3-colorable) as the proof generalizes to any other component. Denote the subgraph of $Q[1,7)$ consisting only of the component containing $1$ as $Q'[1,7)$.
    
    Note that by Lemma \ref{Lem:Max_Divisor}, for any rational number $q$ in $Q'[1,7)$, there is a finite sequence of rational numbers $1=q_0, q_1, q_2, \ldots, q_n=q$ such that for each $0\leq i<n$, either $q_i/q_{i+1}$ or $q_{i+1}/q_i\in \{2,3,4,5,6\}$. Hence, we may write $q$ as $q=2^a \cdot 3^b \cdot 5^c$ for some $a,b,c\in \Z$.

    Let $\hat{i}, \hat{j}, \hat{k}$ be the unit vectors at angles $\theta = 0, 2\pi/3, 4\pi/3$ to the positive $x$-axis, respectively, and then place the rational number $q=2^a3^b5^c$ at $a\hat{i}+b\hat{j}+c\hat{k}$. Note that in this case, the $(x,y)$-coordinate of $q$ is $\big(a-\frac{1}{2}(b+c), \frac{\sqrt{3}}{2}(b-c)\big)$.
    
    We now prove that no two rational numbers in $Q'[1,7)$ coincide at any point $a\hat{i}+b\hat{j}+c\hat{k} \in \R^2$. Suppose $q,r$ are both rational numbers in $Q'[1,7)$ such that $q=2^{a_q}3^{b_q}5^{c_q}$ and $r=2^{a_r}3^{b_r}5^{c_r}$, and $a_q\hat{i}+b_q\hat{j}+c_q\hat{k} = a_r\hat{i}+b_r\hat{j}+c_r\hat{k}$.

    If $a_q=a_r$, then $\frac{1}{2}(b_q+c_q)=\frac{1}{2}(b_r+c_r)$ and $\frac{\sqrt{3}}{2}(b_q-c_q)=\frac{\sqrt{3}}{2}(b_r-c_r)$. Solving this system, we find $b_q=b_r$ and $c_q=c_r$, so $q=r$.

    If $a_q\neq a_r$, then without loss of generality, let $a_q> a_r$, so $a_q=a_r+n$ for some $n\in \Z_{>0}$. Then $n+\frac{1}{2}(b_q+c_q)=\frac{1}{2}(b_r+c_r)$ and $\frac{1+\sqrt{3}}{2}(b_r-c_r)=\frac{1+\sqrt{3}}{2}(b_r-c_r)$. Solving this system, we find $b_q=b_r-n$ and $c_q=c_r-n$. Thus
    \[
    q=2^{a_q}3^{b_q}5^{c_q} = 2^n\cdot 2^{a_r} \cdot 3^{-n}\cdot 3^{b_r} \cdot 5^{-n} \cdot 5^{c_r} = \left(\frac{2}{15}\right)^n r.
    \]
    But $r\in [1,7)$ and $n\in \Z_{>0}$, so $q = \left(\frac{2}{15}\right)^n r < \left(\frac{2}{15}\right)^1 7 = \frac{14}{15} < 1$, a contradiction as $q\in [1,7)$. Thus if $q,r$ share the same coordinate in $\R^2$, $q=r$.

    Now observe that if $q<r$ are adjacent in $Q'[1,7)$, then $r/q\in \{2,3,4,5,6\}$. Observe that in $\R^2$, then the edges and neighbors of $r$ in $Q'[1,7)$ constitute a subgraph of the following image:

    
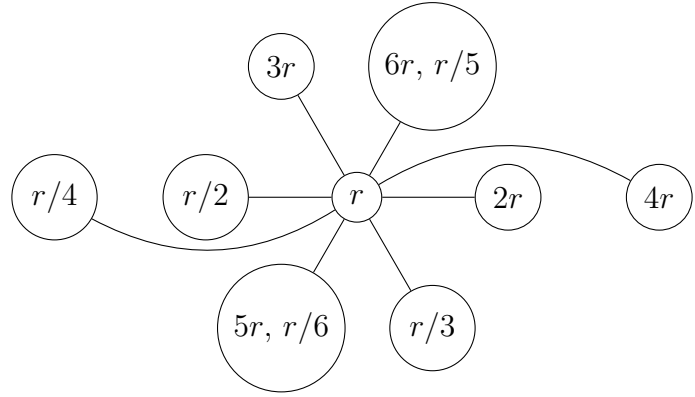
\begin{center}
        \begin{tikzpicture}[scale=1, every node/.style={circle, draw, minimum size=5mm}]
        \node (r) at (0,0) {$r$};
        \node (2r) at (2,0) {$2r$};
        \node (3r) at (-1,1.73) {$3r$};
        \node (4r) at (4,0) {$4r$};
        \node (5r) at (-1,-1.73) {$5r$, $r/6$};
        \node (6r) at (1,1.73) {$6r$, $r/5$};
        \node (r2) at (-2,0) {$r/2$};
        \node (r3) at (1,-1.73) {$r/3$};
        \node (r4) at (-4,0) {$r/4$};

        \draw (r) -- (2r);
        \draw (r) -- (3r);
        \draw[bend left=30] (r) to (4r);
        \draw (r) -- (5r);
        \draw (r) -- (6r);
        \draw (r) -- (r2);
        \draw (r) -- (r3);
        \draw[bend left=30] (r) to (r4);
        \end{tikzpicture}
    \end{center}

    Thus, ignoring for the moment the edges from $r$ to $4r$ and $r/4$, the graph of $Q'[1,7)$ will be a subgraph of the triangulation of the plane with vertices a distance 1 apart. To prove $Q'[1,7)$ is planar, we must prove that if $r$ and $4r$ are vertices of $Q'[1,7)$, then there is always a path from $r$ to $4r$ which does not cross any other edges in $Q'[1,7)$. Consider now the possible neighbors of $2r$ where $r,2r,4r \in Q'[1,7)$, as illustrated in the following graph. Since $r\in[1,7)$, if $kr \in [1,7)$, then it is clear that $k<7$. Further, if $k<1$ and $kr,4r\in[1,7)$, so $4r/kr = 4/k < 7$, so $4/7 < k$. Thus, while the graph on the left includes several possible multiples of $r$, the graph on the right has dropped the values of $r$ that cannot be in $Q'[1,7)$. In particular, $2r$ connects to neither $r/2$ nor $8r$.

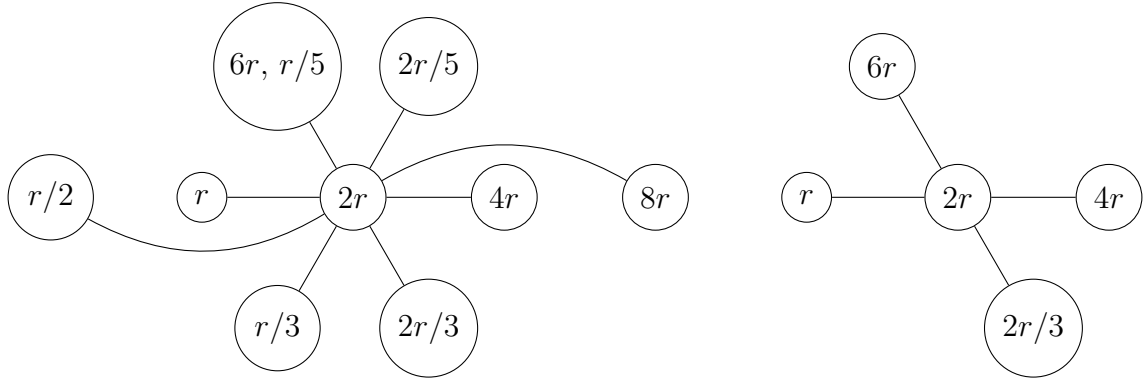
    \begin{center}
        \begin{tikzpicture}[scale=1, every node/.style={circle, draw, minimum size=5mm}]
        \node (r) at (-3,0) {$r$};
        \node (2r) at (-1,0) {$2r$};
        \node (4r) at (1,0) {$4r$};
        \node (6r) at (-2,1.73) {$6r$, $r/5$};
        \node (10r) at (-2,-1.73) {$r/3$};
        \node (12r) at (0,1.73) {$2r/5$};
        \node (20r) at (0,-1.73) {$2r/3$};
        \node (8r) at (3,0) {$8r$};
        \node (r2) at (-5,0) {$r/2$};

        \draw (2r) -- (r);
        \draw (2r) -- (4r);
        \draw (2r) -- (6r);
        \draw (2r) -- (10r);
        \draw (2r) -- (12r);
        \draw (2r) -- (20r);
        \draw[bend left=30] (2r) to (r2);
        \draw[bend left=30] (2r) to (8r);
        
        \node (s) at (5,0) {$r$};
        \node (2s) at (7,0) {$2r$};
        \node (4s) at (9,0) {$4r$};
        \node (6s) at (6,1.73) {$6r$};
        \node (20s) at (8,-1.73) {$2r/3$};

        \draw (2s) -- (s);
        \draw (2s) -- (4s);
        \draw (2s) -- (6s);
        \draw (2s) -- (20s);
        \end{tikzpicture}
    \end{center}

    Thus, in order for $Q'[1,7)$ to be planar, it suffices to prove that either $6r\not\in Q'[1,7)$ or $2r/3 \not\in Q'[1,7)$. Suppose $6r\in Q'[1,7)$. Then $6r<7$, so $r<7/6$. But then $2r/3 < 7/9 < 1$, a contradiction as $r\in [1,7)$. Thus at least one of these edges is not in $Q'[1,7)$, so we can connect $r$ to $4r$ on the side of $2r$ which is missing the edge. Thus $Q'[1,7)$ is planar, and hence $Q[1,7)$ is planar.

    To see that $Q'[1,7)$ is 3-colorable, it suffices to note that the following can easily be extended to a 3-coloring of an infinite graph with colors $\{a,b,c\}$, and that $Q'[1,7)$ is a subgraph of this infinite extension.
    
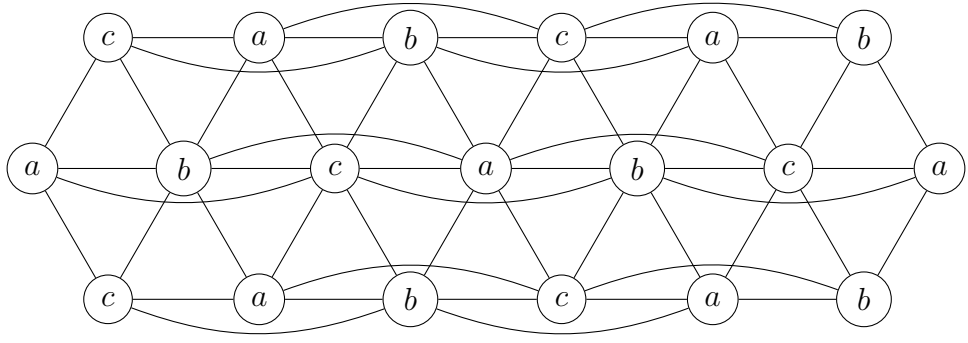
\begin{center}
        \begin{tikzpicture}[scale=1, every node/.style={circle, draw, minimum size=2mm}]
        \node (1r) at (-6,0) {$a$};
        \node (2r) at (-4,0) {$b$};
        \node (3r) at (-2,0) {$c$};
        \node (4r) at  (0,0) {$a$};
        \node (5r) at  (2,0) {$b$};
        \node (6r) at  (4,0) {$c$};
        \node (7r) at  (6,0) {$a$};
        \node (1q) at (-5,1.73) {$c$};
        \node (2q) at (-3,1.73) {$a$};
        \node (3q) at (-1,1.73) {$b$};
        \node (4q) at  (1,1.73) {$c$};
        \node (5q) at  (3,1.73) {$a$};
        \node (6q) at  (5,1.73) {$b$};
        \node (1s) at (-5,-1.73) {$c$};
        \node (2s) at (-3,-1.73) {$a$};
        \node (3s) at (-1,-1.73) {$b$};
        \node (4s) at  (1,-1.73) {$c$};
        \node (5s) at  (3,-1.73) {$a$};
        \node (6s) at  (5,-1.73) {$b$};

        \draw (1q) -- (2q); \draw (2q) -- (3q); \draw (3q) -- (4q); \draw (4q) -- (5q); \draw (5q) -- (6q);
        \draw (1r) -- (2r); \draw (2r) -- (3r); \draw (3r) -- (4r); \draw (4r) -- (5r); \draw (5r) -- (6r); \draw (6r) -- (7r);
        \draw (1s) -- (2s); \draw (2s) -- (3s); \draw (3s) -- (4s); \draw (4s) -- (5s); \draw (5s) -- (6s);
        \draw (1r) -- (1s); \draw (1r) -- (1q);
        \draw (2r) -- (1q); \draw (2r) -- (2q); \draw (2r) -- (1s); \draw (2r) -- (2s);
        \draw (3r) -- (2q); \draw (3r) -- (3q); \draw (3r) -- (2s); \draw (3r) -- (3s);
        \draw (4r) -- (3q); \draw (4r) -- (4q); \draw (4r) -- (3s); \draw (4r) -- (4s);
        \draw (5r) -- (4q); \draw (5r) -- (5q); \draw (5r) -- (4s); \draw (5r) -- (5s);
        \draw (6r) -- (5q); \draw (6r) -- (6q); \draw (6r) -- (5s); \draw (6r) -- (6s);
        \draw (7r) -- (6q); \draw (7r) -- (6s);

        \draw[bend right=20] (1r) to (3r);
        \draw[bend right=20] (3r) to (5r);
        \draw[bend right=20] (5r) to (7r);
        \draw[bend left=20] (2r) to (4r);
        \draw[bend left=20] (4r) to (6r);
        
        \draw[bend right=20] (1s) to (3s);
        \draw[bend right=20] (3s) to (5s);
        \draw[bend left=20] (2s) to (4s);
        \draw[bend left=20] (4s) to (6s);
        
        \draw[bend right=20] (1q) to (3q);
        \draw[bend right=20] (3q) to (5q);
        \draw[bend left=20] (2q) to (4q);
        \draw[bend left=20] (4q) to (6q);
        \end{tikzpicture}
    \end{center}
\end{proof}

\begin{thm}
    $Q[1,7]$ is planar and 3-colorable.
\end{thm}
\begin{proof}
    Observe that the only vertex in $Q[1,7]$ that is not in $Q[1,7)$ is 7. Since every connected component of $Q[1,7)$ is planar, we prove that adding every edge that connects to 7 does not add a $K_5$ or $K_{3,3}$ minor.

    Observe that the only elements adjacent to $7$ are $1, 7/6, 7/5, 7/4, 7/3, 7/2$. Note that the component of $Q[1,7)$ containing $7/6$ also contains $7/3$ and $7/2$. Further, since that component contains $7/2$, it must also contain $7/4$. To see that it contains $7/5$ as well, consider the path $7/4, 21/4, 21/20, 21/5, 7/5$. Thus, adding the vertex $7$ and all the edges from $7$ to the aforementioned vertices merges at most 2 connected components.

    Note that for every rational $q\in (1,7)$, both $7q,q/7 \not\in [1,7]$, so the only edge in $Q[1,7]$ representing a multiplication or division by 7 is the edge from 1 to 7. Thus any rational in the component of $Q[1,7)$ containing $7/2$ is of the form $7\cdot 2^a3^b5^c$ while every rational in the component of $Q[1,7)$ containing $1$ is of the form $2^a3^b5^c$. Thus in $Q[1,7]$, any path from the component containing $7/2$ to the component containing $1$ must cancel the 7 in the numerator, hence there must be a division by 7. Thus the path must contain the edge from 1 to 7, hence said edge must be a bridge connecting two components, where the component containing 1 is a subgraph of the planar graph $Q[1,7)$ and the component containing 7 is a subgraph of $Q(1,7]$. Note also that $Q(1,7] \cong Q[1,7)$ planar via the graph isomorphism $f:Q(1,7] \to Q[1,7)$ given by $f(r) = 7/r$ for any vertex $r$. Since every edge of $K_5$ and $K_{3,3}$ belongs to a cycle, the component of $Q[1,7]$ containing 1 and 7 (i.e. $Q'[1,7]$) is planar since both components it was formed from were planar, and adding the edge from 1 to 7 does not add any new cycles, hence it cannot have introduced a $K_5$ or $K_{3,3}$ minor. Thus $Q'[1,7]$ is planar, so $Q[1,7]$ is planar.

    Further, since the edge from 1 to 7 is a bridge, choose a 3-coloring of the component of $Q[1,7)$ containing 1 and a 3-coloring of the component of $Q(1,7] \cong Q[1,7)$ containing 7 such that the color of 1 is different to the color of 7. Color all other components according to the coloring of $Q[1,7)$.
\end{proof}

We have proven that for any integer $n$, $G[n,7n]$ is planar by proving that $Q[1,7]$ is planar. We now prove that $7$ is the maximum of all $r\in \R$ such that $G\left[n,\lfloor rn\rfloor\right]$ is planar for every integer $n$. We do so by proving that for any $\epsilon>0$, $Q[1,7+\epsilon)$ is non-planar by showing there is a $K_{3,3}$ minor in $Q[1,7+\epsilon)$ for any $\epsilon>0$, but first we need the following technical lemma.

\begin{lemma}\label{Lem:Irrational_Dense}
    For $\alpha$ a positive irrational number and $\epsilon>0$, there exist $a,b\in \Z_{>0}$ such that $a-b\alpha \in (0, \epsilon)$.
\end{lemma}
\begin{proof}
    Let $N\in \Z$ be such that $1/N < \epsilon$. For $i\in\{1, 2, \ldots, N, N+1\}$ let $a_i$ be the minimal integer such that $a_i - i\alpha > 0$. (Note that the sequence $(a_i : 1\leq i\leq N+1)$ is a non-decreasing sequence.) Thus $\{a_i-i\alpha : 1\leq i\leq N+1\}$ is a set of $N+1$ irrational numbers in the interval $[0,1]$. By the pigeonhole principle, there is an interval $[m/N, (m+1)/N]$ such that there are $i < j$ with $a_i - i\alpha, a_j - j\alpha \in [m/N, (m+1)/N]$. Note that $(a_j-a_i) - (j-i)\alpha \in [-1/N, 1/N]$ for any such $i<j$. Furthermore, as it is irrational, $(a_j-a_i) - (j-i)\alpha \in (-1/N, 0) \cup (0, 1/N)$.
    
    If $(a_j-a_i) - (j-i)\alpha \in (0, 1/N) \subset (0,\epsilon)$, then $a=a_j-a_i > (j-i)\alpha > 0$ and $b=j-i>0$ are the integers proving the result.

    If $(a_j-a_i) - (j-i)\alpha \in (-1/N, 0)$, then there is some positive integer $M$ such that $M[(a_j-a_i) - (j-i)\alpha] \in (-1,-1+1/N)$. Thus $a=M(a_j-a_i)+1 \geq 1$ and $b=M(j-i)$ are the integers proving the result.
\end{proof}

\begin{thm}
    $Q[1,7+\epsilon)$ is non-planar as it contains a $K_{3,3}$ minor.
\end{thm}
\begin{proof}
    Without loss of generality, let $\epsilon<1$. We first prove the following fact: Given a rational number $r>1$ and $\epsilon/7>0$, there is a sequence of rational numbers $r=r_0, r_1, \ldots, r_n$ such that
    \begin{enumerate}
        \item $r_i\geq 1$ for all $i$,
        \item $r_{i+1}/r_{i} \in \{2, \frac{1}{3}, \frac{1}{2}\}$ for all $i$,
        \item $r_n \in [1,1+\epsilon/7)$.
    \end{enumerate}
    Observe that $r_0=r$ satisfies the conditions. Assume $r_0, r_1, \ldots, r_i$ satisfy (1) and (2). If $r_i\geq 3$, let $r_{i+1}=r_i/3$, otherwise let $r_{i+1}=2r_i$. (Note that once $r_i\in [1,6]$, all subsequent $r_j\in [1,6]$ as well.) We now prove that there is eventually an $n\in \N$ such that $r_n\in [1,1+\epsilon/7)$ after a slight adjustment of this recursion.

    Let $\alpha = \ln 3/\ln 2 = \log_2(3)$, and note $\alpha$ is irrational. Then by Lemma \ref{Lem:Irrational_Dense}, there are $a,b\in\Z_{>0}$ such that
    \begin{align*}
        a-b\alpha &\in \left( 0, \frac{\log_2(1+\epsilon/7)}{\log_2(r)} \right) \\
        r \frac{2^a}{2^{b\alpha}} &\in (1, 1+\epsilon/7) \\
        r \frac{2^a}{3^b} &\in (1,1+\epsilon/7).
    \end{align*}

    To see that $r2^a/3^b$ can be attained by a sequence, observe that for our recursive definition the exponent of $3$ in the denominator only ever increases, and does so by 1 each time it increases. As such, there is a point in the above recursive definition of a sequence that we reach a term $r2^q/3^b$. If this term is $r2^a/3^b$, we are done. If $a>q$, then we have $1 \leq r 2^q/3^b < r2^a/3^b < 1+\epsilon/7$, and we are done. If $a<q$, we divide by 2 a total of $q-a$ times. Thus there is a sequence as described which eventually terminates in a term in $[1,1+\epsilon/7)$.

    We now prove that $Q[1,7+\epsilon)$ contains the following $K_{3,3}$ minor by explicitly stating the nine paths necessary to construct the minor.
    
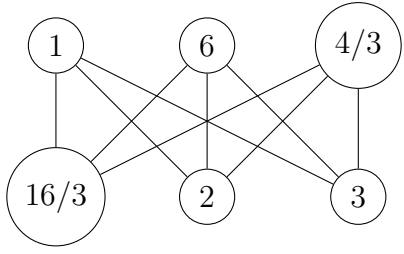
\begin{center}
        \begin{tikzpicture}[scale=1, every node/.style={circle, draw, minimum size=5mm}]
        \node (1r) at (-2,1) {$1$};
        \node (2r) at (0,1) {$6$};
        \node (3r) at (2,1) {$4/3$};
        \node (5r) at (-2,-1) {$16/3$};
        \node (6r) at (0,-1) {$2$};
        \node (7r) at (2,-1) {$3$};
        
        \draw (1r) -- (5r);
        \draw (1r) -- (6r);
        \draw (1r) -- (7r);
        \draw (2r) -- (5r);
        \draw (2r) -- (6r);
        \draw (2r) -- (7r);
        \draw (3r) -- (5r);
        \draw (3r) -- (6r);
        \draw (3r) -- (7r);
        \end{tikzpicture}
    \end{center}

    The paths from $1$ to $2$, $1$ to $3$, $6$ to $2$, $6$ to $3$, and $4/3$ to $16/3$ are clear as these paths are simply an edge.

    The path from $4/3$ to $2$ is $4/3, 4, 2$.

    The path from $4/3$ to $3$ is $4/3, 20/3, 5/3, 5, 5/4, 15/4, 15/8, 45/8, 9/8, 9/2, 3/2, 3$.

    The path from $6$ to $16/3$ is $6, 6/5, 24/5, 8/5, 16/5, 16/15, 16/3$.

    The only path that remains to be shown is the path from $1$ to $16/3$ which does not overlap with any of the aforementioned paths.

    Beginning at $16/3$, we follow the recursive definition of a sequence beginning at $r_0=16/3$ and ending in the interval $[1,1+\epsilon/7)$ where (except for a few final steps) $r_{i+1} = r_i/3$ if $r_i\geq 3$ and $r_{i+1}=2r_i$ if $r_i<3$. Observe that this sequence begins $16/3, 16/9, 32/9, 32/27, \ldots$ and eventually reaches some term $2^{a_0}/3^{b_0}\in (1,1+\epsilon/7)$ with $a_0, b_0 > 0$. Then $7\cdot 2^{a_0}/3^{b_0} \in (7,7+\epsilon)\subset [1,7+\epsilon)$.

    We now follow the following recursive rule to continue this sequence. Given $r_n$ a rational number in the sequence beginning at or after $7\cdot 2^{a_0}/3^{b_0}$, define
    \[
        r_{n+1} = \begin{cases}
            3r_n & \text{if } r_n \leq 7/3, \\
            r_n/2 & \text{if } r_n > 7/3.
        \end{cases}
    \]
    Note that if $r_n\in [1,7+\epsilon)$, then $r_{n+1}\in [1,7+\epsilon)$. Further note that if $r_n = 7\cdot 2^a/3^b$, then $r_{n+1}$ is either $7\cdot 2^{a-1}/3^b$ or $7\cdot 2^a/3^{b-1}$. Let $q$ be the first rational number obtained in this sequence such that either the exponent of $2$ or the exponent of $3$ is $0$, i.e., either $q=7\cdot 2^\alpha$ or $q=7/3^\beta$.
    
    If $q=7\cdot 2^\alpha \in [1, 7+\epsilon)$, then $\alpha=0$ and the sequence
    \[
    \frac{4}{3}, \frac{8}{3}, \ldots, \frac{2^{a_0}}{3^{b_0}}, \frac{7\cdot 2^{a_0}}{3^{b_0}}, \frac{7\cdot 2^{a_0-1}}{3^{b_0}}, \ldots, q=7
    \]
    is a path from $4/3$ to $7$.

    Similarly, if $q=7/3^\beta \in [1, 7+\epsilon)$, then $\beta=0$ or $\beta=1$. Thus $q\in\{7/3, 7\}$ and the sequence
    \[
    \frac{4}{3}, \frac{8}{3}, \ldots, \frac{2^{a_0}}{3^{b_0}}, \frac{7\cdot 2^{a_0}}{3^{b_0}}, \frac{7\cdot 2^{a_0-1}}{3^{b_0}}, \ldots, \frac{7}{3}, 7, 1
    \]
    is a path from $16/3$ to $1$. Further, every rational number in both of these two sequences has a denominator which is a power of $3$. The only possible overlaps with a previous path are $4/3$, $20/3$, or $5/3$. The latter two have a multiple of $5$ in the numerator, so we are independent from those points, and before we introduce the $7$ in the numerator, the power of $2$ in the numerator only increases, so we never go from $16$ to $4$. Thus this path is independent of all previous paths.

    Thus we have found a $K_{3,3}$ minor in $Q[1,7+\epsilon)$ for an arbitrary $\epsilon>0$, so $Q[1,7+\epsilon)$ is non-planar.
\end{proof}

Thus $G[n,7n]$ is planar, and $7$ is the greatest possible real number $r$ for which we can ensure $G[n,rn]$ is planar.

\begin{remark}
    Up to now, not much has been shown about the planarity of divisor graphs. Some open questions include the following. Are there nice algebraic properties of rings that ensure the full divisor graph of the ring is planar and/or non-planar? Some simple conditions are that if the ring does not have at least 5 elements, then it must be planar.

    What other subsets $S\subseteq \Z$ ensure $G(S)$ is planar? For example, if $P$ is the set of primes in $\Z$, then $G(P)$ is planar. Are there any useful classes of ``planar subsets'' of $\Z$?

    There are a few other questions that are still open regarding the planarity of a given divisor graph $G[n,m]$. We know that for $G[n,m]$ to be planar, we must have $m/n < 15$. However, for a given $n$, we currently do not have a way to calculate the maximum ratio $m/n$ for which $G[n,m]$ is planar besides brute force. From our limited computations, it seems that most of the time it lies close to 7.

    We also suspect but have not proven that if $G[n,m]$ contains a $K_5$ minor, then it also has a $K_{3,3}$ minor; that is, $K_5$ is never the first nonplanar subgraph.

    While we proved that $Q[1,7]$ is 3-colorable, and in fact has chromatic number 3, the chromatic number of $Q[1,7+\epsilon)$ seems to be a more difficult question as there are more paths connecting 1 to 7 than just one. In \cite{al2012further}, they prove that divisor graphs are perfect, which would suggest that the chromatic number remains 3 for $\epsilon<1$, there are no cliques of size 4 in $Q[1,7+\epsilon)$. However, they use the Strong Perfect Graph Theorem (see \cite{chudnovsky2006strong}) which is proven for \emph{finite} graphs, yet $Q[1,7+\epsilon)$ is infinite.
\end{remark}

\bibliographystyle{plain} 
\bibliography{refs} 

\end{document}